\documentclass[10pt]{amsart}

\usepackage{amssymb,latexsym, mathtools,tikz}
\usepackage{enumerate}
\usepackage{graphicx}
\usepackage{float}
\usepackage{placeins}
\usepackage{mdframed}
\usepackage{amssymb}
\usepackage{esint}
\usepackage{cool}
\usepackage[all,cmtip]{xy}
\usepackage{mathtools}
\usepackage{amstext} 
\usepackage{array}   
\usepackage[shortlabels]{enumitem}
\usepackage{ytableau}
\usepackage{tikz}
\usetikzlibrary{arrows,decorations.markings, cd}
\tikzstyle arrowstyle=[scale=1]
\tikzstyle directed=[postaction={decorate,
decoration={markings,mark=at position .65 with {\arrow[arrowstyle]{stealth}}}}]
\usepackage{mathdots}
\usepackage{quiver}

\newcolumntype{L}{>{$}l<{$}} 
\newcolumntype{C}{>{$}c<{$}}
\newtheorem{theorem}{Theorem}[section]
\newtheorem{lemma}[theorem]{Lemma}
\newtheorem{cor}[theorem]{Corollary}
\newtheorem{prop}[theorem]{Proposition}
\newtheorem{setup}[theorem]{Setup}
\theoremstyle{definition}
\newtheorem{definition}[theorem]{Definition}
\newtheorem{example}[theorem]{Example}

\newtheorem{obs}[theorem]{Observation}
\newtheorem{notation}[theorem]{Notation}

\newtheorem{chunk}[theorem]{}
\theoremstyle{remark}
\newtheorem{remark}[theorem]{Remark}

\newtheorem{the context}[theorem]{The Context}

\numberwithin{equation}{theorem}
\numberwithin{equation}{section}

\newcommand{\cat}[1]{\mathcal{#1}}

\newcommand{\Span}{\operatorname{Span}}

\newcommand{\coker}{\operatorname{Coker}}

\newcommand{\im}{\operatorname{Im}}

\newcommand{\Ker}{\operatorname{Ker}}

\newcommand{\bbz}{\mathbb{Z}}
\newcommand{\bbn}{\mathbb{N}}

\renewcommand{\geq}{\geqslant}
\renewcommand{\leq}{\leqslant}
\renewcommand{\ker}{\Ker}

\newcommand{\ch}{\textrm{CH}}

\newcommand{\gl}{\operatorname{{GL}}}

\newcommand{\maps}[5]{\xymatrix{#1 \ar[r]^-{#3} & #2 \\
#4 \ar@{|->}[r] & #5 \\}}

\def\w{\wedge}

\def\im{\operatorname{im}}

\newcommand{\F}{\cat{F}}
\renewcommand{\ch}{\operatorname{ch}}

\newcommand{\sgn}{\operatorname{sgn}}

\newcommand{\mdeg}{\operatorname{mdeg}}

\begin{document}
\title[GL-Equivariant Complex]{A GL-Equivariant Complex Inducing Character Identities for Schur Modules}

\author{Keller VandeBogert}
\address{University of Notre Dame}
\email{kvandebo@nd.edu}
\date{\today}

\maketitle

\begin{abstract}
    In this paper we construct a GL-equivariant complex of Schur modules over a ring of positive characteristic that can be used to deduce classical alternating sum identities for Schur polynomials. This complex globalizes to a complex of vector bundles and can also be used to give an explicit construction of an exact sequence predicted by work of Grayson involving Adams operations identities on the algebraic K-theory of a given scheme $X$. The more general complex gives an explicit construction that reproves the aforementioned Adams operations identities in full generality. 
\end{abstract}

\section{Introduction}

Schur polynomials are special classes of polynomials parametrized by partitions of an integer $n$ that arise in the context of combinatorics and representation theory. There is a large body of research devoted to proving identities relating Schur polynomials to other combinatorial classes of polynomials, with more well-known examples given by the Jacobi-Trudi, Giambelli, and Cauchy formulas. The fact that any linear combination of symmetric polynomials can be rewritten in terms of Schur polynomials is well-known, since the Schur polynomials form a basis for the vector space generated by all symmetric polynomials (see, for instance, \cite[Proposition 2.2.10]{grinberg2020hopf}). That being said, it is a generally nontrivial problem to explicitly determine the coefficients appearing in such identities.

Let $p_n$ denote the $n$th-power symmetric polynomial (that is, $p_n = \sum_i x_i^n$); since $p_n$ is symmetric, it may be written as a linear combination of Schur polynomials. More explicitly, there is an equality
\begin{equation}\label{eqn:liuPoloId}
    p_n = \sum_{i=0}^{n-1} (-1)^i s_{(n-i,1^i)},
\end{equation}
where $s_{(n-i,1^i)}$ denotes the Schur polynomial corresponding to the partition $(n-i,1^i)$. This identity may be proved using a variety of techniques, almost all of which are completely combinatorial in nature (see \cite[Proposition 1.1]{grinberg2020petrie} for a discussion of such methods). Given that the identity \ref{eqn:liuPoloId} is an alternating sum of character polynomials for Schur modules, one is tempted to ask if there is a complex of vector spaces inducing the identity \ref{eqn:liuPoloId}. 

In this paper, we answer this question in the affirmative. More precisely, we construct a $\gl (V)$-equivariant complex of vector spaces
\begin{equation}\label{eqn:theEquivariantComplex}
    0 \to S_n (V) \to S_{(n-1,1)} (V) \to \cdots \to S_{(n-i,1^i)} (V) \to \cdots \to \bigwedge^n V \to 0
\end{equation}
over a field of positive characteristic that can be used to deduce the equality \ref{eqn:liuPoloId}. Notice that the existence of such a complex is not immediately obvious: in characteristic $0$, the irreducibility of these Schur modules implies that there are \emph{no} nontrivial equivariant maps $S_{(n-i,1^i)} (V) \to S_{(n-i-1,1^{i+1})} (V)$ for any $0 \leq i \leq n-1$, so the characteristic assumption is necessary. 

This complex also globalizes to the level of vector bundles, whence a special case of \ref{eqn:theEquivariantComplex} shows that for any vector bundle $\F$ over a field of characteristic $p >0$, there is an induced exact sequence of vector bundles
\begin{equation}\label{eqn:FpBundleComplex}
    0 \to F^* \F \to S_p (\F) \to \cdots \to S_{(p-i,1^i)} (\F) \to \cdots \to \bigwedge^p \F \to 0,
\end{equation}
where $F$ denotes the absolute Frobenius morphism. The existence of such an exact sequence was predicted by work of Grayson \cite{grayson1992adams} involving Adams operations identities on the algebraic K-theory of vector bundles on a scheme $X$; a dualized version of this complex was implicit in work of Carter and Lusztig \cite[Discussion 4.3]{carter1974modular}, but the version presented in this paper has the advantage of being generalized to composite integers and with explicit and simple formulas for the differentials. We also use a more general version of \ref{eqn:FpBundleComplex} to give a novel proof of the previously mentioned Adams operations identities established by Grayson.

The paper is organized as follows. In Section \ref{sec:backgroundAndNotation}, we introduce background and notation related to Schur modules, multidegrees, and character polynomials. We particularly emphasize the case of Schur modules associated to hook partitions, since these are the modules we will be most interested in later sections.

In Section \ref{sec:complexAndHomology}, we construct the complex of vector spaces \ref{eqn:equivComplex} and compute its cohomology explicitly. The definition of the differentials is relatively straightforward, being a composition of naturally defined comultiplication and multiplication maps; the main content is that these maps descend to well-defined maps on the appropriate Schur modules. We then prove that the cohomology of this complex is isomorphic to Frobenius powers of other Schur modules, and use this fact to give an inductive proof of the identity \ref{eqn:liuPoloId} by performing a multigraded rank count on the appropriate complex.

Finally, in Section \ref{sec:Ktheory} we consider some applications to the algebraic K-theory of vector bundles. We first recall some elementary facts of K-theory and Adams operations, including an explicit form for Adams operations proved by Grayson \cite{grayson1992adams}. As previously mentioned, the complex \ref{eqn:FpBundleComplex} gives an explicit construction of a complex predicted by work of Grayson and may also be used to reprove Grayson's Adams operation identities on the algebraic K-theory of vector bundles. 

\section{Background and Notation}\label{sec:backgroundAndNotation}

In this section, we introduce some of the necessary background that will be needed for later sections. Most of the material here will only be needed for Schur modules associated to hook partitions, but we still state many definitions in more generality. We also talk about formal Frobenius powers and character polynomials for torus-invariant modules (equivalently, this is just taking the sum of the multidegrees appearing in any multigraded basis). 

To begin, we recall the definition of a partition along with some associated notation.

\begin{definition}
A \emph{partition} $\lambda = (\lambda_1 \geq \cdots \lambda_n)$ is a finite sequence of nonincreasing integers. The partition $\lambda$ is a partition of an integer $N$ if $|\lambda| := \lambda_1 + \cdots + \lambda_n = N$. The notation $(1^i)$ will denote the partition 
$$\underbrace{(1 , \dots , 1)}_{i \ \textrm{times}}.$$
Any partition of the form $(a,1^b)$ for some integers $a$ and $b$ is called a \emph{hook} partition. 
\end{definition}

\begin{remark}
The ``hook" terminology comes from the standard representation of partitions as Young tableaux, since the tableau corresponding to a hook partition resembles a hook.
\end{remark}

\begin{notation}
Given any partition $\lambda$, the notation $S_\lambda (V)$ denotes the Schur module corresponding to $\lambda$. 
\end{notation}

Our conventions for Schur modules are chosen such that $S_{(d)} (V) = S_d (V)$ (the symmetric algebra) and $S_{(1^d)} (V) = \bigwedge^d V$ (the exterior algebra). For an introduction to Schur modules from a commutative algebraic perspective, see \cite{akin1982schur} or Chapter $2$ of \cite{weyman2003}. Throughout this paper, we will assume familiarity with the standard bialgebra structure on the symmetric and exterior algebras; for a concrete and explicit description of the corresponding operations, see Chapter I, Section $1$ of \cite{akin1982schur}.

\begin{notation}\label{not:beginningNot}
Let $V$ be a vector space over a field of positive characteristic $p$; assume $\dim V = n$ and $v_1 , \dots , v_n$ is a basis for $V$. Given an indexing set $I = (\ell_1 < \cdots < \ell_i)$ and an exponent vector $\alpha = (\alpha_1 , \dots , \alpha_n)$ with $|\alpha| = j$, define
$$v_I := v_{\ell_1} \w \cdots \w v_{\ell_i} \in \bigwedge^i V, \quad v^\alpha := v_1^{\alpha_1} \cdots v_n^{\alpha_n} \in S_j (V).$$
The notation $\epsilon_i$ denotes the vector with a $1$ in the $i$th spot and $0$s elsewhere. 
\end{notation}

\begin{chunk}\label{chunk:SchurDiscussion}
For the convenience of the reader, we give a more explicit description of Schur modules associated to hook partitions; that is, Schur modules of the form $S_{(a,1^b)} (V)$ for some vector space $V$. By taking homogeneous strands of the Koszul complex, there is an induced \emph{tautological Koszul complex}
$$ \cdots \to  \bigwedge^b V \otimes S_a (V) \xrightarrow{\kappa_{a,b}}  \bigwedge^{b-1} V \otimes S_{a+1} (V)  \to \cdots$$
whose maps can be described explicitly as the composition
\begingroup\allowdisplaybreaks
\begin{align*}
    \bigwedge^b V \otimes S_a (V) &\xrightarrow{\Delta \otimes 1} \bigwedge^{b-1} V \otimes V \otimes S_{a} (V) \\
    &\xrightarrow{1 \otimes m} \bigwedge^{b-1} V \otimes S_{a+1} (V).
\end{align*}
\endgroup
where $\Delta$ and $m$ denote the appropriate comultiplication and multiplication, respectively. Then, the Schur module $S_{(a,1^b)}$ is equal to
$$\ker \kappa_{a,b} = \im \kappa_{a-1,b+1} = \coker \kappa_{a-2,b+2}.$$
\end{chunk}

\begin{remark}\label{rk:abuseNotation}
By an abuse of notation, elements of the Schur module $S_{(i,1^j)} V$ will often be denoted as simply $v_I \otimes v^\alpha$, where $|I| = i+1$ and $|\alpha| = j-1$. This should cause no confusion since $S_{(i,1^j)} V$ is equivalently described as a quotient of $\bigwedge^{i+1} V \otimes S_{j-1} V$ by the discussion given in \ref{chunk:SchurDiscussion}.
\end{remark}

\begin{definition}
Let $V$ be any vector space and $m \in \bbn$ some integer. Then the \emph{$m$th formal Frobenius power} $F^m : S_i (V) \to S_{im} (V)$ is defined to be the $k$-linear map induced by sending $v_i \mapsto v_i^m$. 
\end{definition}

\begin{remark}
Since there is no mention of the characteristic of the field $k$ nor the primality of the integer $m$, the map $F^m$ is only defined on the standard basis monomials of $S_\bullet (V)$ and extended by linearity. 
\end{remark}

\begin{example}
If $m=4$, $i=2$, and $V$ is a two dimensional space, then
$$F^4 S_2 (V) = \Span_k \{ v_1^8 , v_1^4v_2^4, v_2^8 \}.$$
\end{example}

\begin{definition}\label{def:mgr}
Adopt notation as in Notation \ref{not:beginningNot}. Then the \emph{multidegree} of a basis element $v_I \otimes v^\alpha \in \bigwedge^a V \otimes S_b (V)$, denoted $\mdeg$, is defined as
$$\mdeg (v_I \otimes v^\alpha ) := v_I \cdot v^\alpha \in \bbz [v_1 , \dots , v_n],$$
where the notation $v_I \in \bbz [v_1 , \dots , v_n]$ denotes the element $v_{j_1} \cdots v_{j_i}$. Observe that this multigrading descends to a well-defined multigrading on the Schur modules $S_\lambda (V)$, since the straightening relations preserve the multigrading.

Any arbitrary element is \emph{multigraded} if it is a linear combination of basis elements with the same multidegree.
\end{definition}

The multigrading of Definition \ref{def:mgr} should cause no confusion, since it is induced by the natural choice of multigrading on the symmetric and exterior algebras $S_\bullet V$ and $\bigwedge^\bullet V$. Equivalently, a multigraded element is an element generating a $1$-dimensional torus-invariant subspace of the corresponding Schur module, and the multigrading is simply the multigraded character of this subspace.

\begin{definition}\label{def:character}
Let $\lambda$ be any partition and let $B \subset E \subset S_{\lambda} (V)$ denote any choice of multigraded basis for some multigraded subspace of $S_\lambda (V)$. Then the \emph{character} of the subspace $E$ is the sum
$$\ch (E) := \sum_{b \in B} \mdeg (b) \in \bbz [v_1 , \dots , v_n].$$
The character polynomial of a Schur module $S_\lambda (V)$ is often called the \emph{Schur polynomial} associated to the partition $\lambda$.
\end{definition}

\begin{remark}
For the reader familiar with representation theory, in the case that $E$ is a GL-equivariant subspace, Definition \ref{def:character} is equivalent to the more standard definition of the character as the sum
$$\sum_{\chi} \dim (S_\lambda(V)_\chi) v^\chi,$$
where the sum is taken over all characters of the standard maximal torus of $\gl (V)$, the notation $S_\lambda (V)_\chi$ denotes the weight space corresponding to the character $\chi$, and the notation $v^\chi$ denotes $v_1^{\chi_1} \cdots v_n^{\chi_n}$ (assuming $\dim V = n$). 
\end{remark}

\begin{example}
If $m$ is any integer and $V$ is an $n$-dimensional vector space on basis $v_1 , \dots , v_n$, then
$$\ch (F^m V) = v_1^m + v_2^m + \cdots + v_n^m.$$
Likewise, if $V$ is a $3$-dimensional vector space on basis $v_1 , v_2 , v_3$, then it is an easy exercise to verify
$$\ch (S_{(2,1)} (V)) = {v}_{1}^{2}{v}_{2}+{v}_{1}{v}_{2}^{2}+{v}_{1}^{2}{v}_{3}+2\,{v}_{1}{v
       }_{2}{v}_{3}+{v}_{2}^{2}{v}_{3}+{v}_{1}{v}_{3}^{2}+{v}_{2}{v}_{3}^{2},$$
       $$\ch (S_{(2,1,1)} (V)) = {v}_{1}^{2}{v}_{2}{v}_{3}+{v}_{1}{v}_{2}^{2}{v}_{3}+{v}_{1}{v}_{2}{v}_{
       3}^{2}.$$
\end{example}

We conclude this section with a proposition that will be useful in the next section. Intuitively, this proposition says that one can perform straightening relations on any multigraded element of $S_{(i,1^j)} (V)$ to ensure that a fixed $v_i$ appears with the same exponent in each element of the support (using the notational abuse mentioned in Remark \ref{rk:abuseNotation}).

\begin{prop}\label{prop:preserveExponent}
Let $f \in S_{(i,1^j)} V$ be any multigraded element. Then for any $1 \leq \ell \leq n$, the element $f$ may be represented by a linear combination of tableaux $v_{I_k} \otimes v^{\alpha^k}$ with $\alpha^k_\ell = N$, for all $k$, where $N$ is some fixed integer.
\end{prop}

\begin{proof}
Suppose the exponent of $v_\ell$ in the multidegree of $f$ is $N+1$. By definition of the multidegree, all tableaux $v_{I_k} \otimes v^{\alpha^k}$ appearing in the support of $f$ have the property that $\alpha_\ell^k = N$ or $N+1$. If $\alpha_\ell^k = N+1$ for some $k$, then $\ell \notin I_k$ and we may use the straightening relations to rewrite $v_{I_k} \otimes v^{\alpha^k}$ in terms of tableaux with the exponent of $v_\ell$ being reduced by $1$.
\end{proof}

\section{The Complex and Its Cohomology}\label{sec:complexAndHomology}

In this section, we construct the complex (\ref{eqn:theEquivariantComplex}) mentioned in the introduction and compute its cohomology explicitly (see Theorem \ref{thm:exactEquivariantComplex}). As mentioned in the introduction, the existence of nontrivial equivariant maps $S_{(m-i,1^i)} (V) \to S_{(m-i-1,1^{i+1})} (V)$ is a phenomenon unique to the positive characteristic case. We begin this section by defining these maps explicitly and proving that they are well-defined. Note:

\begin{remark}
All results in this section hold when vector spaces over a field $k$ are replaced with free modules over a commutative ring $R$. 
\end{remark}

\begin{definition}\label{def:thePhiMap}
Let $V$ be a vector space over a field of positive characteristic $p$; assume $\dim V = n$ and $v_1 , \dots , v_n$ is a basis for $V$. Define $\phi$ to be the map:
\begingroup\allowdisplaybreaks
\begin{align*}
    \bigwedge^i V \otimes S_j V &\xrightarrow{1 \otimes \Delta} \bigwedge^i V \otimes V \otimes S_{j-1} V \\
    &\xrightarrow{m \otimes 1} \bigwedge^{i+1} V \otimes S_{j-1} V ,
\end{align*}
\endgroup
where $\Delta$ denotes comultiplication in the symmetric algebra and $m$ denotes multiplication in the exterior algebra. More explicitly, $\phi$ acts on an element as so:
$$\phi (v_I \otimes v^\alpha) = \sum_{i=1}^n \alpha_i v_I \w v_i \otimes v^{\alpha - \epsilon_i}.$$
\end{definition}

\begin{prop}
Let $m$ be any integer divisible by $p$. The map $\phi$ of Definition \ref{def:thePhiMap} descends to a well-defined map
$$\phi : S_{(m-i, 1^i)} V \to S_{(m-i-1,1^{i+1})} V.$$
\end{prop}

\begin{proof}
Recall that the Schur module $S_{(m-i,1^i)}$ is obtained as the cokernel of the natural map
$$\psi: \bigwedge^{i+2} V \otimes S_{m-i-2} V \to \bigwedge^{i+1} V \otimes S_{m-i-1} V,$$
whence it suffices to show that the following diagram commutes:
$$\xymatrix{\bigwedge^{i+1} V \otimes S_{m-i-1} V \ar[r]^-{\psi}  \ar[d]^-{\phi} & \bigwedge^{i} V \otimes S_{m-i} V \ar[d]^-{\phi}  \\
\bigwedge^{i+2} V \otimes S_{m-i-2} V \ar[r]^-{\psi} & \bigwedge^{i+1} V \otimes S_{m-i-1} V.}$$
This is a straightforward computation; let $I = (j_1 < \cdots < j_{i+1})$ be an indexing set of size $i+1$ and $\alpha$ an exponent vector with $|\alpha| = m-i-1$. Going clockwise around the diagram:
\begingroup\allowdisplaybreaks
\begin{align*}
    v_I \otimes v^\alpha &\mapsto \sum_{i \in I} \sgn(i \in I) v_{I \backslash i} \otimes v^{\alpha+ \epsilon_i} \\
    &\mapsto \sum_{i \in I} \sum_{j \neq i} \sgn (i \in I) \alpha_j v_{I \backslash i} \w v_j \otimes v^{\alpha+ \epsilon_i - \epsilon_j} \\
    &+ \sum_{i \in I} (\alpha_i + 1) v_I \otimes v^\alpha .
\end{align*}
\endgroup
Moving counterclockwise:
\begingroup\allowdisplaybreaks
\begin{align*}
    v_I \otimes v^\alpha &\mapsto \sum_{j \notin I} \alpha_j v_I \w v_j \otimes v^{\alpha - \epsilon_j} \\
    &\mapsto -\sum_{j \notin I} \sum_{\substack{i \in I \\
    i \neq j}} \sgn( i \in I) \alpha_j v_{I \backslash i} \w v_j \otimes v^{\alpha + \epsilon_i - \epsilon_j} \\
    &-\sum_{i \notin I} \alpha_i v_I \otimes v^\alpha. 
\end{align*}
\endgroup
Comparing both of the above, notice first that both of the terms $\sum_{i \in I} \sum_{j \neq i} \sgn (i \in I) \alpha_j v_{I \backslash i} \w v_j \otimes v^{\alpha+ \epsilon_i - \epsilon_j}$ and $-\sum_{j \notin I} \sum_{\substack{i \in I \\
    i \neq j}} \sgn( i \in I) \alpha_j v_{I \backslash i} \w v_j \otimes v^{\alpha + \epsilon_i - \epsilon_j}$ are $0$ since we may sum over all $j < i$ and relabel (this is identical to the proof that the Koszul complex is a complex). For the leftover terms, simply observe that
    \begingroup\allowdisplaybreaks
    \begin{align*}
         \sum_{i \in I} (\alpha_i + 1) + \sum_{i \notin I} \alpha_i &= |\alpha| + |I| \\
         &= m-i-1 + i+1 = 0 \quad \textrm{since} \ p \mid m.
    \end{align*}
    \endgroup
\end{proof}

\begin{remark}
By an abuse of notation, we will make no distinction between the map $\phi$ and the induced map on Schur modules.
\end{remark}

\begin{cor}\label{cor:seqOfMaps}
There is a well defined sequence of maps
$$0 \to S_m(V)  \to S_{(m-1,1)} V \to \cdots \to \bigwedge^m V \to 0,$$
where each map is induced by $\phi$.
\end{cor}

\begin{prop}\label{prop:theyAreComplex}
The sequence of maps
$$0 \to S_m(V) \to S_{(m-1,1)} V \to \cdots \to \bigwedge^m V \to 0$$
of Corollary \ref{cor:seqOfMaps} forms a complex.
\end{prop}

\begin{proof}
The map $\phi$ is equivalently described as the dual to multiplication $\bigwedge^i V \otimes D_j (V) \to \bigwedge^{i-1} V \otimes D_{j+1} (V)$ (where $D_\bullet (V)$ denotes the divided power algebra) by the trace element $t \in V^* \otimes V$, so $\phi^2 = 0$. Descending to a quotient does not change this fact.
\end{proof}

Combining the above results implies that the following definition is indeed well-defined.

\begin{definition}
Given a vector space $V$ over a field of characteristic $p>0$ and an integer $m$ divisible by $p$, let $N_m (V)$ denote the complex of Proposition \ref{prop:theyAreComplex}.
\end{definition}

\begin{remark}
We will use the convention that the complex $N_m (V)$ is indexed cohomologically; that is, $N_m (V)_0 = S_m (V)$ and $N_m (V)_m = \bigwedge^m V$. 
\end{remark}

As it turns out, the complex $N_m (V)$ is more than just a complex of vector spaces: it is also a complex of $\gl (V)$-modules.

\begin{prop}
The differentials of $N_m (V)$ are $\gl (V)$-equivariant.
\end{prop}

\begin{proof}
The map $\phi : \bigwedge^i V \otimes S_j (V) \to \bigwedge^{i+1} V \otimes S_{j-1} (V)$ is $\gl (V)$-equivariant for all $i$, $j$ (since it is a composition of equivariant maps). Again, descending to the quotient does not change this fact (since in this case, the subcomplex we are quotienting by is also $\gl (V)$-invariant).
\end{proof}

Now, the remainder of this section will be dedicated to computing the cohomology of the complex $N_m (V)$ for a given $m$. 

\begin{obs}
Let $V' := V/ kv_\ell = \bigoplus_{j \neq \ell}^n kv_j$. Then for all integers $i$, $j$ there is a split short exact sequence of vector spaces:
$$0 \to \bigwedge^{i-1} V' \otimes S_{j-1} (V) \to S_{(j,1^i)} V  \to S_{(j,1^i)} V' \to 0,$$
where the left map is the composition
$$\bigwedge^{i-1} V' \otimes S_{j-1} (V) \xrightarrow{v_\ell \w -} \bigwedge^i V \otimes S_{j-1} V \to S_{(j,1^i)} V,$$
and the right map is induced by the projection $V \to V'$.
\end{obs}

\begin{setup}\label{set:exactnessSet}
Let $V' := V/ kv_\ell = \bigoplus_{j \neq \ell}^n kv_j$. Let $L_m (V,v_\ell)$ denote the complex
$$0 \to v_i^{m-1} \hookrightarrow S_{m-1}V \to V' \otimes S_{m-2} V \to \cdots \to \bigwedge^{m-1} V' \otimes V \to \bigwedge^m V' \to 0$$
where the differential $d$ is induced by multiplication by the element $r \in V' \otimes V^*$ corresponding to the projection $V \to V'$. Let $s \in (V')^* \otimes V$ be the element corresponding to the inclusion $V' \hookrightarrow V$ and let $h$ be the map induced by multiplication by $s$.
\end{setup}

\begin{obs}
Adopt notation and hypotheses as in Setup \ref{set:exactnessSet}. Then there is a short exact sequence of complexes
$$0 \to L_m (V,v_\ell) \xrightarrow{v_\ell \w -} N_m (V) \to N_m (V') \to 0,$$
where the left map is left multiplication by $v_\ell$ and the right map is induced by the projection $V \to V'$.
\end{obs}

The following proposition will be essential for computing the cohomology of $N_m (V)$.

\begin{prop}\label{prop:theHtpy}
Adopt notation and hypotheses as in Setup \ref{set:exactnessSet}. Let $v_I \otimes v^\alpha \in \bigwedge^{i} V' \otimes S_{m-i-1} V \subset L_m (V,v_\ell)$ and assume that $\alpha_\ell + 1$ is not divisible by $p$. Then 
$$v_I \otimes v^{\alpha} = \frac{-1}{(\alpha_\ell +1)} \Big( d h(v_I \otimes v^\alpha ) + h d (v_I \otimes v^\alpha) \Big).$$
\end{prop}

\begin{proof}
Let us first compute the composition $d h$:
\begingroup\allowdisplaybreaks
\begin{align*}
    v_I \otimes v^\alpha &\mapsto \sum_{j \in I} \sgn (j) v_{I \backslash j} \otimes v^{\alpha + \epsilon_j} \\
    &\mapsto \sum_{j \in I} \sum_{k \neq \ell ,j}\sgn (j) \alpha_k v_{I \backslash j} \w v_k \otimes v^{\alpha + \epsilon_j - \epsilon_k} \\
    &+ \sum_{j \in I} (\alpha_j+1) e_I \otimes v^\alpha. 
    \end{align*}
    \endgroup
    Next, let us compute the composition $h d$:
    \begingroup\allowdisplaybreaks
    \begin{align*}
    v_I \otimes v^\alpha & \mapsto \sum_{\substack{k \neq \ell , \\
    k \notin I}} \alpha_k v_I \w v_k \otimes v^{\alpha- \epsilon_k} \\
    &\mapsto -\sum_{\substack{k \neq \ell, \\
    k \notin I}} \sum_{j \in I } \sgn(j) \alpha_k v_{I \backslash j} \w v_k \otimes v^{\alpha - \epsilon_k + \epsilon_j} \\
    &+ \sum_{\substack{k \neq \ell , \\
    k \notin I}} \alpha_k v_I \otimes v^\alpha. \\
\end{align*}
\endgroup
Adding both of the above together, we obtain a constant multiple of $v_I \otimes v^\alpha$ with coefficient
\begingroup\allowdisplaybreaks
\begin{align*}
    \sum_{j \in I} (\alpha_j + 1) + \sum_{j \notin I, \ j \neq \ell} \alpha_j &= -\alpha_\ell + |\alpha| + |I| \\
    &= -\alpha_\ell - 1 . 
\end{align*}
\endgroup
This coefficient is nonzero if and only if $\alpha_\ell + 1$ is not divisible by $p$.
\end{proof}

\begin{obs}\label{obs:htpyDescends}
For each $1 \leq \ell \leq \dim V$, the homotopy of Proposition \ref{prop:theHtpy} induces homotopies $h_\ell$ on the complex $N_m (V)$ (defined as $v_\ell \w h$ where $h$ is as in Setup \ref{set:exactnessSet} for the appropriate value of $\ell$).
\end{obs}



The following definition introduces certain vector spaces that will end up appearing as the cohomology of $N_m (V)$.

\begin{definition}\label{def:frobeniusSchurModule}
Let $\lambda$ be a partition. The notation $S^p_\lambda (V)$ will denote the quotient of $S_\lambda (V)$ by the vector subspace generated by all basis elements with multidegree \emph{not} in the algebra $S_\bullet (F^p V)$. 
\end{definition}

\begin{remark}
Notice that Definition \ref{def:frobeniusSchurModule} is well-defined since the straightening relations on the Schur modules $S_\lambda (V)$ preserve multidegree. Notice moreover that if $|\lambda|$ is not divisible by $p$, then $S_\lambda^p (V) = 0$. 
\end{remark}

\begin{example}
Let $V$ be a $3$-dimensional vector space. Then the module $S_{3,1}^2 (V)$ has basis represented by the standard tableaux
$$ \ytableausetup
{boxsize=1.7em}
\begin{ytableau}
1 & 1 & 2\\
2 \\
\end{ytableau}, \quad \ytableausetup
{boxsize=1.7em}
\begin{ytableau}
1 & 1 & 3\\
3 \\
\end{ytableau}, \quad 
\ytableausetup
{boxsize=1.7em}
\begin{ytableau}
2 & 2 & 3\\
3 \\
\end{ytableau}.$$
In particular,
$$\ch (S_{3,1}^2 (V)) = v_1^2 v_2^2 + v_1^2 v_3^2 + v_2^2 v_3^2 = F^2 \ch (\bigwedge^2 V).$$
\end{example}

The observation that $\ch (S_{2,1,1}^2 (V)) =  F^2 \ch (\bigwedge^2 V)$ is not just a coincidence, and is in fact true in general.

\begin{lemma}\label{lem:frobeniusCharacter}
Let $n$ be any integer divisible by some prime $p$. Then there is an isomorphism of vector spaces
$$\eta : S_{n-i,1^i} (V) \to S^p_{pn-i,1^i} (V)$$
with the property that $F^p (\mdeg (e)) = \mdeg (\eta (e))$ for any multigraded element $e \in S_{n-i,1^i} (V)$. 

In particular, there is an equality
$$\ch (S_{(pn-i,1^i)}^p (V)) = F^p \ch (S_{(n-i,1^i)} (V)).$$
\end{lemma}

\begin{remark}
In the above, we will use the convention that $S_{a,1^i} (V) = 0$ for $a \leq 0$.
\end{remark}

\begin{proof}
Use the notation $\Big( S_{n-i} (V) \otimes \bigwedge^i V \Big)^p$ to denote the subspace of all elements of $S_{n-i} (V) \otimes \bigwedge^i V$ having multidegree in $S_\bullet (F^pV)$. The proof will follow first by defining a map
$$\eta' : S_{n-i} (V) \otimes \bigwedge^i V \to \Big( S_{pn - i} (V) \otimes \bigwedge^i V \Big)^p$$
and showing that $\eta'$ restricts to a well-defined map on the appropriate Schur modules. To this end, define
\begingroup\allowdisplaybreaks
\begin{align*}
    \eta' : S_{n-i} (V) \otimes \bigwedge^i V &\to \Big( S_{pn - i} (V) \otimes \bigwedge^i V \Big)^p , \\
    v^\alpha \otimes v_I &\mapsto v^{p \alpha} \cdot v_I^{p-1} \otimes v_I. 
\end{align*}
\endgroup
In the above, the notation $v_I \in S_i (V)$ denotes the product $v_{j_1} \cdots v_{j_i} \in S_i (V)$, where $I = (j_1 < \cdots < j_i)$ and the notation $v_I^{p-1}$ denotes the element $v_{j_1}^{p-1} \cdots v_{j_i}^{p-1} \in S_{(p-1)i} (V)$. Notice that the above is well-defined since the image of any element under $\eta'$ clearly has multidegree in $S_\bullet (F^p V)$, and the degree of that element is $p(n-i) + (p-1)i + i = np $.  

Recall that the Schur module $S_{(n-i,1^i)} (V)$ is equivalently the image of the natural map $S_{n-i-1} (V) \otimes \bigwedge^{i+1} V \to S_{n-i} (V) \otimes \bigwedge^i V$, whence it suffices to show that the following diagram commutes (where the vertical maps are defined as in the discussion \ref{chunk:SchurDiscussion}):
\[\begin{tikzcd}
	{S_{n-i-1} (V) \otimes \bigwedge^{i+1} V} & {S_{pn-i-1} (V) \otimes \bigwedge^{i+1} V} \\
	{S_{n-i} (V) \otimes \bigwedge^i V} & {S_{pn-i} (V) \otimes \bigwedge^i V}
	\arrow["{\eta'}", from=1-1, to=1-2]
	\arrow["{\eta'}", from=2-1, to=2-2]
	\arrow[from=1-1, to=2-1]
	\arrow[from=1-2, to=2-2]
\end{tikzcd}\]
Moving clockwise around the diagram, one obtains:
\begingroup\allowdisplaybreaks
\begin{align*}
    v^\alpha \otimes v_I &\mapsto v^{p \alpha} \cdot v_I^{p-1} \otimes v_I \\
    &\mapsto \sum_{i \in I} \sgn (i) v^{p \alpha} \cdot v_i^p \cdot v_{I \backslash i}^{p-1} \otimes v_{I \backslash i} \\
    &= \sum_{i \in I} \sgn (i) v^{p(\alpha + \epsilon_i)} \cdot v_{I \backslash i}^{p-1} \otimes v_{I \backslash i}.
\end{align*}
\endgroup
Moving counterclockwise,
\begingroup\allowdisplaybreaks
\begin{align*}
    v^\alpha \otimes v_I &\mapsto \sum_{i \in I} \sgn (i) v^{\alpha + \epsilon_i} \otimes v_{I \backslash i} \\
    &\mapsto \sum_{i \in I} \sgn (i) v^{p(\alpha + \epsilon_i)} \cdot v_{I \backslash i}^{p-1} \otimes v_{I \backslash i}.
\end{align*}
\endgroup
The induced map $\eta$ is evidently invertible, in which case one obtains the desired isomorphism. 
\end{proof}

The following corollary identifies the vector spaces $S_{(a,1^b)}^p (V)$ (non-equivariantly) as Frobenius powers of other Schur modules.

\begin{cor}\label{cor:frobSchurMod}
Let $V$ be a vector space over a field of positive characteristic $p >0$. Then there is a multigraded isomorphism
$$S_{(pm-i,1^i)}^p (V) \cong F^p S_{(m-i,1^i)}.$$
\end{cor}

\begin{proof}
Notice that there is a well-defined map
\begingroup\allowdisplaybreaks
\begin{align*}
    S_{(pm-i,1^i)}^p (V) &\to F^p S_{(m-i,1^i)} (V) \\
\end{align*}
\endgroup
induced by the map $\Big( S_{pm-i} (V) \otimes \bigwedge^i V \Big)^p \to F^p (S_{m-i} (V) \otimes \bigwedge^i V)$ sending an element $v^{p \alpha} \cdot v_I^{p-1} \otimes v_I \mapsto F^p (v^{\alpha} \otimes v_I)$, where $|\alpha| = m-i$ and the notation in the proof of Lemma \ref{lem:frobeniusCharacter} is being used here. This map is multigraded by construction and descends to a well-defined morphism of the corresponding Schur modules by a computation identical to that of Lemma \ref{lem:frobeniusCharacter}.
\end{proof}

Finally, we arrive at the main result of this section. 

\begin{theorem}\label{thm:exactEquivariantComplex}
Adopt notation and hypotheses as in Setup \ref{set:exactnessSet}. Then there is a $\gl (V)$-equivariant isomorphism
$$H^i (N_m (V) ) \cong F^pS_{(m/p-i,1^i)} (V).$$
\end{theorem}

\begin{proof}
We will first prove that there is an isomorphism of \emph{vector spaces} $H^i (N_m (V)) \cong S_{m-i,1^i}^p (V)$. Observe that $S_{m-i,1^i}^p (V)$ is a vector subspace of $H^i (N_m (V))$ for each $i \geq 0$. This follows because, a priori, it is clear that any element of $S_{m-i,1^i}^p (V)$ is contained in $\ker \phi$; since $\phi$ preserves multidegree, no such element can lie in the image of $\phi$ since it would then be the image of an element in $S_{m-i+1,1^{i-1}}^p (V)$, which must be $0$.

It remains to show that any multigraded element in the kernel of $\phi$ that does not have multidegree in $S_\bullet (F^p V)$ must lie in the image of $\phi$. By Proposition \ref{prop:preserveExponent}, any such element may be represented by a linear combination of tableaux $v_{I_k} \otimes v^{\alpha^k}$ where, for some fixed $j$, one has $j \in I_k$ and $\alpha_j +1$ is not divisible by $p$. However, Proposition \ref{prop:theHtpy} implies that any such element is in the image of $\phi$; this establishes that $H^i (N_m (V)) \cong S_{m-i,1^i}^p (V)$ (as vector spaces).

It remains to establish the claim of equivariance. Consider the map
\begingroup\allowdisplaybreaks
\begin{align*}
    \nu : F^p S_{(m/p-i,1^i)} (V) &\to S_{(m-i,1^i)} (V), \quad \textrm{induced by} \\
    F^p (v^\alpha \otimes v_I) &\mapsto v^{p \alpha} v_I^{p-1} \otimes v_I.
\end{align*}
\endgroup
It suffices to show that for any $g \in \gl (V)$ and $F^p (v^\alpha \otimes v_I) \in F^p S_{(m/p-i,1^i)} (V)$, one has
$$\nu (g \cdot  F^p (v^\alpha \otimes v_I) ) - g \cdot \nu (F^p (v^\alpha \otimes v_I)) \in \im \phi.$$
Recall that one only needs to prove equivariance for elementary row operations of the form $v_i \mapsto v_i + \lambda v_j$ and $v_\ell \mapsto v_\ell$ for some integers $i \neq j$, and $\ell \neq i$ (the fact that $\nu$ commutes with scalar matrix multiplication is clear). We may assume that the given element of $F^p S_{(m-p-i,1^i)} (V)$ is of the form $F^p(v^\alpha \cdot v_i^b \otimes v_i \cdot v_I)$, where $i \notin I$ and $\alpha_i = 0$. One then computes:
\begingroup\allowdisplaybreaks
\begin{align*}
    \nu (g \cdot  F^p (v^\alpha \cdot v_i^b \otimes v_i \cdot v_I) ) &= \nu \Big( F^p (v^{\alpha} (v_i + \lambda v_j)^b \otimes (v_i + \lambda v_j) \w v_I \Big)  \\
    &= \nu \Big( F^p ( \sum_{\ell+k = b} \lambda^k \binom{b}{\ell} v^\alpha v_i^\ell v_j^k \otimes v_i v_I \\
    &\quad  + \sum_{\ell+k = b} \lambda^{k+1} \binom{b}{\ell} v^\alpha v_i^\ell v_j^k \otimes v_j v_I ) \Big) \\
    &= \sum_{\ell+k = b} \lambda^{pk} \binom{b}{\ell}^p v^{p\alpha} v_I^{p-1} v_i^{p(\ell+1)-1} v_j^{pk} \otimes v_i v_I \\
    &\quad + \sum_{\ell+k = b} \lambda^{p(k+1)} \binom{b}{\ell}^p v^{p\alpha} v_I^{p-1} v_i^{p\ell} v_j^{p(k+1)-1} \otimes v_j v_I , \quad \textrm{and}\\ 
    g \cdot \nu (F^p (v^\alpha \cdot v_i^b \otimes v_i \cdot v_I)) &= g \cdot \Big( v^{p \alpha} v_I^{p-1} v_i^{p(b+1)-1} \otimes v_i v_I \Big) \\
    &= \sum_{\ell+k = p(b+1)-1} \lambda^k \binom{p(b+1)-1}{\ell} v^{p \alpha} v_I^{p-1} v_i^\ell v_j^k \otimes v_i v_I \\
    &\quad + \sum_{\ell+k = p(b+1)-1} \lambda^{k+1} \binom{p(b+1)-1}{\ell} v^{p \alpha} v_I^{p-1} v_i^\ell v_j^k \otimes v_j v_I. \\
\end{align*}
\endgroup
We may rewrite the latter terms to be more recognizable as so:
\begingroup\allowdisplaybreaks
\begin{align*}
    &\sum_{\ell+k = p(b+1)-1} \lambda^k \binom{p(b+1)-1}{\ell} v^{p \alpha} v_I^{p-1} v_i^\ell v_j^k \otimes v_i v_I \\
    =&\sum_{\substack{\ell+k = p(b+1)-1 \\
    p \not\vert \ell+1 \ \textrm{or} \ p \not\vert k}}  \lambda^k \binom{p(b+1)-1}{\ell} v^{p \alpha} v_I^{p-1} v_i^\ell v_j^k \otimes v_i v_I \\
    &+ \sum_{\substack{\ell+k = p(b+1)-1 \\
    p \mid \ell+1 \ \textrm{and} \ p \mid k}}  \lambda^k \binom{p(b+1)-1}{\ell} v^{p \alpha} v_I^{p-1} v_i^\ell v_j^k \otimes v_i v_I \\
    =& \sum_{\substack{\ell+k = p(b+1)-1 \\
    p \not\vert \ell+1 \ \textrm{or} \ p \not\vert k}}  \lambda^k \binom{p(b+1)-1}{\ell} v^{p \alpha} v_I^{p-1} v_i^\ell v_j^k \otimes v_i v_I \\
    &+\sum_{\ell+k = b}  \lambda^{pk} \binom{p(b+1)-1}{p(\ell+1)-1} v^{p \alpha} v_I^{p-1} v_i^{p(\ell+1)-1} v_j^{pk} \otimes v_i v_I, \\
\end{align*}
\endgroup
where the final equality comes from writing $\ell+1 = p \ell'$, $k = p k'$, and then reindexing the summation. A similar summation reindexing yields
\begingroup\allowdisplaybreaks
\begin{align*}
    &\sum_{\ell+k = p(b+1)-1} \lambda^{k+1} \binom{p(b+1)-1}{\ell} v^{p \alpha} v_I^{p-1} v_i^\ell v_j^k \otimes v_j v_I \\
    =& \sum_{\substack{\ell+k = p(b+1)-1 \\
    p \not\vert \ell \ \textrm{or} \ p \not\vert k+1}} \lambda^{k+1} \binom{p(b+1)-1}{\ell} v^{p \alpha} v_I^{p-1} v_i^\ell v_j^k \otimes v_j v_I \\
    &+ \sum_{\ell+k = b} \lambda^{p(k+1)} \binom{p(b+1)-1}{p(k+1)-1} v^{p\alpha} v_I^{p-1} v_i^{p\ell} v_j^{p(k+1)-1} \otimes v_j v_I.
\end{align*}
\endgroup
By Proposition \ref{prop:lukasThmCor}, there is an equality $\binom{p(b+1)-1}{p(\ell+1)-1} = \binom{p(b+1)-1}{p(k+1)-1} = \binom{b}{\ell}$, whence combining all of the above equalities yields
\begingroup\allowdisplaybreaks
\begin{align*}
    &\nu (g \cdot  F^p (v^\alpha \cdot v_i^b \otimes v_i \cdot v_I) ) - g \cdot \nu (F^p (v^\alpha \cdot v_i^b \otimes v_i \cdot v_I)) \\
    =& - \sum_{\substack{\ell+k = p(b+1)-1 \\
    p \not\vert \ell+1 \ \textrm{or} \ p \not\vert k}}  \lambda^k \binom{p(b+1)-1}{\ell} v^{p \alpha} v_I^{p-1} v_i^\ell v_j^k \otimes v_i v_I \\
    &- \sum_{\substack{\ell+k = p(b+1)-1 \\
    p \not\vert \ell \ \textrm{or} \ p \not\vert k+1}} \lambda^{k+1} \binom{p(b+1)-1}{\ell} v^{p \alpha} v_I^{p-1} v_i^\ell v_j^k \otimes v_j v_I.
\end{align*}
\endgroup
The latter term is an element of $\ker \phi$ that is not contained in $S_{(m-i,1^i)}^p (V)$, and hence must lie in the image of $\phi$ by the argument at the beginning of the proof. This yields equivariance.
\end{proof}

The following proposition was used in the proof of Theorem \ref{thm:exactEquivariantComplex}; we state and prove it here for convenience.

\begin{prop}\label{prop:lukasThmCor}
Let $m$ and $n$ be any two integers and $p>0$ any prime integer. Then,
$$\binom{pm + p-1}{pn + p-1} \equiv \binom{m}{n} \mod p.$$
\end{prop}

\begin{proof}
Write $m = \sum_i m_i p^i$ and $n = \sum_i n_i p^i$ in their respective base $p$ expansions. Then there are induced base $p$ expansions $pm + p-1 = \sum_{i > 0} m_{i-1} p^i + p-1$ and $pn + p-1 = \sum_{i>0} n_{i-1} p^i + p-1$. By Lucas's Theorem for binomial coefficients, there is an equality
$$\binom{pm+p-1}{pn+p-1} \equiv \prod_{i>0} \binom{m_{i-1}}{n_{i-1}} \cdot \binom{p-1}{p-1} \equiv \binom{m}{n} \mod p.$$
\end{proof}

\begin{cor}\label{cor:primeBaseCase}
For any prime $p$, there is an exact $\gl (V)$-equivariant complex
$$0 \to F^p V \to S_p (V) \to S_{p-1,1} (V) \to \cdots \to S_{p-i,1^i} (V) \to \cdots \to \bigwedge^p V \to 0.$$
\end{cor}

\begin{proof}
Simply observe that $S_p^p (V) = F^p (V)$ and $S_{p-i,1^i}^p (V) = 0$ for any $i >0$. 
\end{proof}

\begin{remark}
A result of Liu and Polo \cite[Proposition 1.4.2]{liu2021cohomology} also constructs a complex similar to that of Corollary \ref{cor:primeBaseCase}, although the differentials are not given explicitly. A dualized version of the complex of Corollary \ref{cor:primeBaseCase} with Weyl modules also appears in works of Carter and Lusztig (see \cite[Discussion 4.3, page 234]{carter1974modular}). It is worth noting that in both of these cases, the authors do not consider the analogous complex for composite integers.
\end{remark}

\begin{remark}
Combining Theorem \ref{thm:exactEquivariantComplex} with Corollary \ref{cor:frobSchurMod}, it follows that the complex $N_m (V)$ satisfies
$$H^i (N_m (V)) = F^p (N_{m/p} (V))_i.$$
As we shall see, this equality is well-suited to inductive arguments for properties of the complexes $N_m (V)$.
\end{remark}

We conclude this section by stating the induced character identity alluded to in the introduction.

\begin{cor}\label{cor:theCharacterIdentity}
For any integer $m \geq 1$, there is an equality:
$$\ch (F^m (V)) = \sum_{i=0}^{m-1} (-1)^i \ch (S_{(m-i,1^i)} (V) ).$$
\end{cor}

\begin{remark}
Notice that by definition
$$p_m = \ch (F^m (V)), \quad s_{(m-i,1^i)} = \ch (S_{(m-i,1^i)} (V)),$$
in which case Corollary \ref{cor:theCharacterIdentity} is precisely the equality \ref{eqn:liuPoloId} stated in the introduction.
\end{remark}

\begin{proof}
Write $m = p_1^{\alpha_1} \cdots p_\ell^{\alpha_\ell}$ for some primes $p_r$, $1 \leq r \leq \ell$. The proof is by induction on $|\alpha| = \alpha_1 + \cdots + \alpha_\ell$. If $|\alpha| = 0$ then the statement is trivial, and if $|\alpha| = 1$ then $m$ is prime and the statement follows from Corollary \ref{cor:primeBaseCase}. 

Assume now that $|\alpha| \geq 2$ and let $p$ denote any prime dividing $m$. By Theorem \ref{thm:exactEquivariantComplex}, there is the following equality of characters:
$$\ch (S_m (V)) - \ch (S_m^p (V)) = \sum_{i=1}^{m-1}(-1)^{i+1}  \Big( \ch ( S_{(m-i,1^i)} (V) ) - \ch ( S_{(m-i,1^i)}^p (V)) \Big).$$
Rearranging the above, one obtains the equality
$$\sum_{i=0}^{m-1} (-1)^i \ch (S_{(m-i,1^i)} (V) ) = \sum_{i=0}^{m-1} (-1)^i \ch (S_{(m-i,1^i)}^p (V) ),$$
and by Lemma \ref{lem:frobeniusCharacter}, $\ch (S_{(m-i,1^i)}^p (V)) = F^p \ch (S_{(m/p - i , 1^i)} (V))$. By the inductive hypothesis applied to $m/p$, one obtains
$$\sum_{i=0}^{m-1} (-1)^i \ch (S_{(m-i,1^i)}^p (V) ) = F^p \ch (F^{m-p} (V)) = \ch (F^m V),$$
whence
$$\ch (F^m (V)) = \sum_{i=0}^{m-1} (-1)^i \ch (S_{(m-i,1^i)} (V) ).$$
\end{proof}

\section{Applications to the Algebraic K-Theory of Vector Bundles}\label{sec:Ktheory}

In this section, we consider the globalized version of the complex $N_m (V)$ for vector bundles over some scheme $X$ and some of the consequences this complex and its cohomology have for the K-theory of $X$. We first begin the section by recalling some definitions related to K-theory.

\begin{definition}
Let $R$ be a Noetherian ring or $X$ be a scheme. Then the \emph{Grothendieck groups} $K_0 (R)$ and $K_0 (X)$ are the groups formally generated by equivalence classes of projective modules (respectively vector bundles on $X$) $[P]$ such that $[P] = [P'] + [P'']$ if there exists a short exact sequence
$$0 \to P' \to P \to P'' \to 0.$$
Notice that $K_0 (X)$ admits the structure of a ring by defining $[P]\cdot [P'] := [P \otimes P]$ (where the tensor product is taken over $R$ or the structure sheaf $\cat{O}_X$).
\end{definition}

\begin{definition}
Let $(F,d^F)$ be a complex of finitely generated projective $R$-modules or vector bundles over a scheme $X$.
\begin{enumerate}
    \item The \emph{Euler characteristic} $\chi (F)$ is defined to be
    $$\chi (F) := \sum_i (-1)^i [F_i] \in K_0 (R) \quad (\textrm{or} \ K_0 (X) ).$$
    \item The \emph{secondary Euler characteristic} $\chi' (F)$ is defined to be
    $$\chi' (F) := \sum_i (-1)^i [\im d_i^F] \in K_0 (R) \quad (\textrm{or} \ K_0 (X) ).$$
\end{enumerate}
\end{definition}

\begin{remark}\label{rk:alternativeEulerChar}
Notice that one also has the equality
$$\chi (F) = \sum_i (-1)^i [H_i (F) ].$$
\end{remark}

Throughout the rest of this section, let $X$ denote a scheme and $K_0 (X)$ the Grothendieck group associated to $X$. The remainder of the results will be stated for the algebraic K-theory of vector bundles with the tacit understanding that all of these results hold for $K_0 (R)$, where $R$ is a ring.

\begin{definition}
Let $X$ be a scheme. Then the Adams operations on $K_0 (X)$ are a collection of maps $\{ \psi^k \}_{k \in \bbz}$ satisfying the following:
\begin{enumerate}
    \item For each $k \in \bbz$, the map $\psi^k : K_0 (X) \to K_0 (X)$ is a ring homomorphism.
    \item For every $\ell$, $k \in \bbz$, there is an equality $\psi^k \circ \psi^\ell = \psi^{k \ell}$. 
    \item If $\cat{L}$ is a line bundle, then $\psi^k [\cat{L}] = [\cat{L}^{\otimes k}]$. 
\end{enumerate}
\end{definition}

The fact that Adams operations exist can be deduced from the \emph{splitting principle} (see \cite{may2005note}), but this technique is often not conducive toward explicit computations. Luckily, the following result due to Grayson \cite{grayson1992adams} gives a much more explicit form for Adams operations. In the following statement, the notation $CP$ is shorthand for the mapping cone of the identity $P \to P$, and $S_k$ of a complex is the standard degree $k$ tautological Koszul complex on $P$ (see the discussion of \ref{chunk:SchurDiscussion}).

\begin{prop}[{\cite{grayson1992adams}}]\label{prop:graysonIdentity}
Let $\psi^k$ denote the $k$th Adams operation on $K_0 (X)$. Then,
$$\psi^k [P] = \chi' (S_k CP).$$
In particular, there is the following equality in $K_0 (X)$:
$$\psi^k [P] = \sum_{i=0}^{k-1} (-1)^i [S_{(k-i,1^i)} (P) ].$$
\end{prop}

\begin{remark}
If $X$ is a scheme over a field $k$ of characteristic $p >0$, then $X$ is equipped with the absolute Frobenius map $F : X \to X$. Recall that this map acts as the identity of the space $X$ and is the $p$th power map on the structure sheaf. 

Let $P$ denote a vector bundle on $X$ as above. Then the $p$th Adams operation is precisely the Frobenius pullback $F^*$; that is,
$$\psi^p [P] = [F^* P].$$
\end{remark}

The following is the main result of this section, and gives a complex of vector bundles for which the equality of Proposition \ref{prop:graysonIdentity} may be deduced from an explicit complex of vector bundles.

\begin{theorem}
Let $X$ be a scheme over a field $k$ of characteristic $p > 0$ and $m$ any integer divisible by $p$. Given a vector bundle $\cat{F}$ on $X$, there is a complex of vector bundles:
\begin{equation}\label{eqn:equivComplex}
N_m (\cat{F}) : \quad 0 \to S_{m} (\cat{F}) \to S_{m-1,1} (\cat{F})  \to \cdots \to S_{m-i,1^i}(\cat{F}) \to \cdots \to \bigwedge^{m} \cat{F} \to 0,\end{equation}
satisfying $H_i (N_m (\cat{F}) ) = F^* S_{(m/p - i ,1^i)} (\cat{F})$. In particular, this complex induces the following equality in $K_0 (X)$:
$$\psi^{m} [\cat{F}] = \sum_{i=0}^{m-1} (-1)^i [S_{(m-i,1^i)} (\cat{F}) ].$$
\end{theorem}

\begin{proof}
The complex of \ref{eqn:equivComplex} is simply a globalized version of the complex of Theorem \ref{thm:exactEquivariantComplex}; the equivariance ensures that globalizing is well-defined. The induced identity on the algebraic K-theory of $X$ follows by induction on $|\alpha|$, where $m = p_1^{\alpha_1} \cdots p_k^{\alpha_k}$ in an identical manner to the proof of Corollary \ref{cor:theCharacterIdentity}. For the base case $|\alpha|=1$, the complex of vector bundles of \ref{eqn:equivComplex} induces the exact sequence
$$0 \to F^* \cat{F} \to S_p (\cat{F}) \to \cdots \to S_{(p-i,1^i)} (\cat{F}) \to \cdots \to \bigwedge^p \cat{F} \to 0,$$
which induces the identity
$$\sum_{i=0}^{p-1} (-1)^i [S_{(p-i,1^i)} (\cat{F}) ] = [F^* \cat{F} ] = \psi^p [\cat{F}].$$
For the inductive step, the complex \ref{eqn:equivComplex} combined with Remark \ref{rk:alternativeEulerChar} induces the following equality in $K_0 (X)$:
$$\sum_{i=0}^{m/p-1} (-1)^i [F^* S_{(m/p-i,1^i)} (\cat{F}) ] = \sum_{i=0}^{m-1} (-1)^i [S_{(p^\ell-i,1^i)} (\cat{F}) ],$$
and proceeding inductively, one finds:
\begingroup\allowdisplaybreaks
\begin{align*}
    \sum_{i=0}^{m/p-1} (-1)^i [F^* S_{(m/p-i,1^i)} (\cat{F}) ] &= \psi^p \Big( \sum_{i=0}^{m/p-1} (-1)^i [S_{(m/p-i,1^i)} (\cat{F}) ] \Big) \\
    &= \psi^p \Big( \psi^{m/p} [\cat{F}] \Big) \\
    &= \psi^{m} [\cat{F}].
\end{align*}
\endgroup
\end{proof}

\bibliographystyle{amsalpha}
\bibliography{biblio}
\addcontentsline{toc}{section}{Bibliography}

\end{document}